\documentclass{amsart}
\usepackage[
top=30truemm,
bottom=30truemm,
left=30truemm,
right=30truemm]{geometry}
\usepackage{amsfonts, amssymb, amsmath, verbatim, amsthm, mathrsfs, amscd, enumerate, ascmac, fancyhdr, caption, hyperref, framed,ulem,subfigure}
\usepackage{bbm}
\usepackage{eucal}

\usepackage{xcolor}
\usepackage{tikz}
\usetikzlibrary{positioning,intersections, calc, arrows,decorations.markings}

\newtheorem{theorem}{Theorem}

\newtheorem*{theorem*}{Theorem}

\newtheorem*{conjecture*}{Conjecture}

\newtheorem{corollary}[theorem]{Corollary}

\newtheorem*{claim*}{Claim}
\theoremstyle{definition}

\newtheorem*{goal*}{Goal}

\newcommand{\linearmapj}{L_j}
\newcommand{\linearfamily}{\mathbf{L}}

\theoremstyle{remark}

\def\M{\mathfrak{M}}

\title[Tomographic Fourier Extension Identities]{Tomographic Fourier Extension Identities for Submanifolds of $\mathbb{R}^n$}

\author[Bennett]{Jonathan Bennett}
\address[Jonathan Bennett]{School of Mathematics, The Watson Building, University of Birmingham, Edgbaston, Birmingham, B15 2TT, England.}
\email{J.Bennett@bham.ac.uk}

\author[Nakamura]{Shohei Nakamura}
\address[Shohei Nakamura]{Department of Mathematics, Graduate School of Science, Osaka University, Toyonaka, Osaka 560-0043, Japan.}
\email{srmkn@math.sci.osaka-u.ac.jp}

\author[Shiraki]{Shobu Shiraki}
\address[Shobu Shiraki]{Department of Mathematics, Graduate School of Science and Engineering, Saitama University, Saitama 338-8570, Japan.}
\email{sshiraki@mail.saitama-u.ac.jp}

\begin{document}

\date{\today}


\maketitle
\begin{abstract}
We establish identities for the composition $T_{k,n}(|\widehat{gd\sigma}|^2)$, where $g\mapsto \widehat{gd\sigma}$ is the Fourier extension operator associated with a general smooth $k$-dimensional submanifold of $\mathbb{R}^n$, and $T_{k,n}$ is the $k$-plane transform. Several connections to problems in Fourier restriction theory are presented.
\end{abstract}
\section{Introduction}
The purpose of this article is to establish identities involving expressions of the form 
\begin{equation}\label{composition}
T_{k,n}(|\widehat{gd\sigma}|^2),
\end{equation}
 where $T_{k,n}$ denotes the \textit{$k$-plane transform} on $\mathbb{R}^n$ and $$\widehat{gd\sigma}(x)=\int_Sg(\xi)e^{-2\pi ix\cdot\xi}d\sigma(\xi)$$ is the \textit{Fourier extension operator} associated with a general smooth $k$-dimensional submanifold $S$ of $\mathbb{R}^n$; here $d\sigma$ denotes surface measure on $S$. Many problems in harmonic analysis and its applications call for an understanding of Fourier extension operators, and we refer to \cite{Stovall} for further context. The particular interest in compositions of the form \eqref{composition} stems from two very simple observations. The first is that they involve $L^2$ norms on affine subspaces, allowing for the application of $L^2$ methods (such as Plancherel's theorem). The second is the well-known fact that the $k$-plane transform is invertible on $L^2$, raising the prospect that a variety of expressions involving $|\widehat{gd\sigma}|^2$ may be understood via the composition \eqref{composition}. This idea may be traced back at least as far as Planchon and Vega \cite{PV} (see also \cite{VegaEsc}), and has motivated several works recently -- see for example \cite{BBFGI}, \cite{BI}, \cite{BV}, \cite{BN}. 
 
In this work we place particular emphasis on the \textit{generality of the submanifold} $S$, allowing us to bring the underlying geometric features of the tomographic data $T_{k,n}(|\widehat{gd\sigma}|^2)$  to the fore. Our results build on the work of the first and second authors in \cite{BN} in the particular case where $S$ is the unit sphere. In addition to clarifying the underlying geometry in that work, the generality of our results leads to richer connections with contemporary problems in Fourier restriction theory.

In what follows we use the standard parametrisation of the grassman manifold $\M_{k,n}$ of $k$-planes in $\mathbb{R}^n$ by a $k$-dimensional subspace $\pi$ and a translation parameter $y\in\pi^\perp$. This allows the $k$-plane transform to be written as
\begin{equation*}
T_{k,n} f(\pi,y)
=
\int_{\mathbb R^n}f(x)\delta_{\pi+\{y\}}(x)dx:=\int f(x+y)d\lambda_\pi(x),
\end{equation*}
where $\lambda_\pi$ is Lebesgue measure on $\pi$.

\section{Identities and applications}
Our main result here consists of two simple formulae for $T_{k,n}(|\widehat{gd\sigma}|^2)(\pi,y)$ that hold under the assumption that $S$ and $\pi$ satisfy a certain transversality condition. Notably, we see that $T_{k,n}$ is unable to distinguish $|\widehat{gd\sigma}|^2$ from $T_{n-k,n}^*\mu$ for a large family of measures $\mu$ on $\M_{n-k,n}$, again provided a suitable transversality condition is satisfied. Here $T_{n-k,n}^*$ denotes the adjoint $(n-k)$-plane transform on $\mathbb{R}^n$, 
\begin{equation}\label{e:KPlane}
T_{n-k,n}^*\mu(x)
=
\int_{\M_{n-k,n}}\delta_{\theta+\{z\}}(x)d\mu(\theta,z),
\end{equation}
which is of course a superposition of Dirac masses on $(n-k)$-planes. As we shall see shortly, this tomographic understanding of $|\widehat{gd\sigma}|^2$ is relevant to several questions in the restriction theory of the Fourier transform. For example (see the forthcoming Section \ref{Sect:stein}), it sheds some light on a variant of a longstanding conjecture of Stein \cite{S} concerning the manner in which Kakeya-type maximal operators might control Fourier extension operators.  In the setting of rather general submanifolds of $\mathbb{R}^n$ this (somewhat tentative) conjecture, which arose in discussions with Tony Carbery some years ago, states that if $S$ is a smooth (compact) $k$-dimensional submanifold of $\mathbb{R}^n$, then 
\begin{equation}\label{genSt}
\int_{\mathbb{R}^n}|\widehat{gd\sigma}(x)|^2w(x)dx\lesssim\int_S |g(\xi)|^2\sup_{y\in T_\xi S}T_{n-k,n}w((T_\xi S)^\perp,y)d\sigma(\xi)
\end{equation}
for all weight functions $w$.
A manifestly weaker form of \eqref{genSt}, generalising a conjecture attributed to Mizohata and Takeuchi (see \cite{BRV}), is the claim that 
\begin{equation}\label{genMT}
\int_{\mathbb{R}^n}|\widehat{gd\sigma}(x)|^2w(x)dx\lesssim 
\|T_{n-k,n}w\|_\infty\|g\|_2^2.
\end{equation}
We refer to \cite{BCSV}, \cite{CIW}, \cite{BBC3}, \cite{BN} and \cite{BGNO} for some further discussion and recent results  relating to \eqref{genSt} and \eqref{genMT}, which are commonly stated in the case that $S$ is a sphere. 

The family of measures $\mu$ with the claimed property, which of course depends on the function $g$, is constructed via measures $\nu$
on the tangent bundle $$TS=\{(\xi,y):\xi\in S, y\in T_\xi S\},$$ and specifically those whose pushforwards $\nu_S$ under the natural projection map $TS\ni(\xi,y)\mapsto \xi\in S$ are given by 
\begin{equation}\label{marg}
d\nu_S(\xi)=|g(\xi)|^2d\sigma(\xi).
\end{equation} 
Perhaps the simplest such measure $\nu$ is $d\nu(y,\xi)=\delta_0(y)|g(\xi)|^2dyd\sigma(\xi)$.
Finally, the family of measures $\mu$ that we seek consists of  the pushforwards of such $\nu$ under the (Gauss) map $$TS\ni(\xi,y)\mapsto ((T_\xi S)^\perp,y)\in\M_{n-k,n},$$ so that
\begin{align*}
\int_{\M_{n-k,n}}\varphi d\mu
&=
\int_{TS}\varphi((T_\xi S)^\perp,y) d\nu(\xi,y).
\end{align*}
We note that the support of $\mu$ is contained in the set of translates of the normal planes to $S$. 
\begin{theorem}\label{mainid}
Suppose that $S$ is a $k$-dimensional $C^1$ submanifold of $\mathbb{R}^n$ and $\pi$ is a $k$-dimensional subspace of $\mathbb{R}^n$ for which
\begin{equation}
T_\xi S\cap\pi^\perp=\{0\}\;\mbox{ for all }\; \xi\in S,
\tag{T}
\end{equation}
and
\begin{equation}
\langle\xi-\eta\rangle\cap\pi^\perp=\{0\}\;\mbox{ for all }\;\xi,\eta\in S.
\tag{GT}
\end{equation}
Then for any measure $\mu$ satisfying the above conditions,
\begin{equation}\label{id}
T_{k,n}(|\widehat{gd\sigma}|^2)(\pi,y)=
\int_S\frac{|g(\xi)|^2}{|(T_\xi S)^\perp\wedge\pi|}d\sigma(\xi)
=
T_{k,n}T_{n-k,n}^*(\mu)(\pi,y).
\end{equation}
\end{theorem}
Before turning to the context of Theorem \ref{mainid} we make some clarifying remarks. 
Regarding the wedge product in \eqref{id}, given any two transverse subspaces $V, W$ of $\mathbb{R}^n$ of complementary dimensions $\ell$ and $m$, we define $|V\wedge W|$ to be $|v_1\wedge\cdots\wedge v_\ell\wedge w_1\wedge\cdots\wedge w_m|$, where $\{v_j\}$, $\{w_j\}$ are orthonormal bases of $V$ and $W$ respectively. Equivalently,
$$
|U\wedge V|:=\int_V\int_Ue^{-\pi|u+v|^2}dudv,
$$
which of course has the advantage of being explicitly well-defined.

The conditions (T) and (GT) are transversality conditions relating $S$ and $\pi$, and both are necessary for \eqref{id} to hold -- a fact that we establish in Section \ref{sect:opt}. As a consequence of this, our expressions for the tomographic data $T_{k,n}(|\widehat{gd\sigma}|^2)$ do not appear to allow $|\widehat{gd\sigma}|^2$ to be reconstructed, at least directly. Nevertheless, as we have indicated earlier, Theorem \ref{mainid} does have a number of direct applications to the theory of extension operators, and we illustrate this with three examples in Sections \ref{SS1}--\ref{SS3} below.

The condition (GT) guarantees that $S$ intersects any translate of $\pi^\perp$ in at most one point, allowing it to be viewed as a graph of a function $\phi$ over $\pi$. The condition (T) stipulates that all tangent spaces to $S$ meet $\pi^\perp$ transversely, and ensures that this function $\phi$ is $C^1$.
Specifically, if $U\subseteq \pi$ is the orthogonal projection of $S$ onto $\pi$, and $u\in U$, then $\phi(u)$ is the unique element of the set $S\cap (\{u\}+\pi^\perp)-\{u\}$. By construction, $\phi:U\rightarrow\pi^\perp$, and 
\begin{equation}\label{graphoverpi}
S=\{u+\phi(u):u\in U\}.
\end{equation}

The conditions (T) and (GT) are closely related. The local condition (T) may be viewed as a limiting (or infinitesimal) form of the global condition (GT) as $\eta$ approaches $\xi$. Under various assumptions on $S$ one of these conditions may be seen to imply the other. For example, if $k=1$ and $S$ is connected (making $S$ a curve), an application of the mean value theorem reveals that (T)$\implies$(GT). In general, however, neither of these conditions implies the other, even if $S$ is connected. Helical surfaces provide simple examples for which (T) is satisfied while (GT) is not. On the other hand, curves in the plane with a point of inflection can satisfy (GT) but not (T). 

Of course if $S$ has dimension or codimension 1, then Theorem \ref{mainid} involves the \textit{Radon transform} $R:=T_{n-1,n}$ and \textit{$X$-ray transform} $X:=T_{1,n}$ respectively. Specifically, for $k=1$ the identity \eqref{id} becomes
\[
X(|\widehat{gd\sigma}|^2)(\omega,v)
=
\int_S\frac{|g(\xi)|^2}{|\tau(\xi)\cdot \omega|} d\sigma(\xi)=XR^*\mu(\omega,v),
\]
where $\tau(\xi)$ denotes a unit tangent to $S$. Here we are identifying one-dimensional subspaces $\pi$ with vectors $\omega\in\mathbb{S}^{n-1}$.
On the other hand, for $k=n-1$ the identity \eqref{id} becomes
\[
R(|\widehat{gd\sigma}|^2)(\omega,t)
=
\int_S\frac{|g(\xi)|^2}{|v(\xi)\cdot \omega|} d\sigma(\xi)=RX^*\mu(\omega,t),
\]
where $v(\xi)$ denotes a unit normal to $S$. Here we have indulged a slightly more serious abuse of notation by reparametrising hyperplanes by a normal vector $\omega\in\mathbb{S}^{n-1}$ and a distance $t$ from the origin.

The first identity in \eqref{id}, and in particular its independence of the translation parameter $y\in\pi^\perp$, may be interpreted as a \textit{conservation law}, generalising certain energy conservation properties of dispersive PDE, such as the time-dependent Schr\"odinger equation. This perspective is best understood in the setting of \textit{parametrised extension operators}, where the submanifold $S$ is parametrised by a $C^1$ injective map $\Sigma:U\rightarrow\mathbb{R}^n$. Here $U$ is a subset of $\mathbb{R}^k$, and the (parametrised) extension operator associated to $\Sigma$ is given by
$$
Ef(x)=\int_U e^{-2\pi ix\cdot\Sigma(\xi)}f(\xi)d\xi,
$$
where $x\in\mathbb{R}^n$. A simple computation reveals that $Ef=\widehat{gd\sigma}$, where $f$ and $g$ are related by $$f(\xi)=\left|\frac{\partial\Sigma}{\partial \xi_1}\wedge\cdots\wedge\frac{\partial\Sigma}{\partial \xi_k}\right|g(\Sigma(\xi)),$$ and the first of the two identities in \eqref{id} becomes
\begin{equation}\label{idpara}
T_{k,n}(|Ef|^2)(\pi,y)=
\int_U\frac{|f(\xi)|^2}{\left|\frac{\partial\Sigma}{\partial \xi_1}\wedge\cdots\wedge\frac{\partial\Sigma}{\partial \xi_k}\wedge\pi^\perp\right|}d\xi.
\end{equation}
In the particular case where $k=n-1$, $U=\mathbb{R}^{n-1}$, $\Sigma(\xi)=(\xi,|\xi|^2)$, and $(x,t)\in\mathbb{R}^{n-1}\times\mathbb{R}$, then $u(x,t):=Ef(x,t)$ solves the \textit{Schr\"odinger equation} $i\partial_tu=\Delta_x u$ with initial data $u(x,0)=\widehat{f}(x)$. Taking $\pi$ to be purely spatial -- that is the horizontal hyperplane at height $t$ -- the identity \eqref{idpara} reduces to the classical energy conservation law $\|u(\cdot,t)\|_{L^2_x}=\|u(\cdot,0)\|_{L^2_x}$. 

The first identity in Theorem \ref{mainid} has some noteworthy points of contact with the recent literature. The variety uncertainty principle in \cite{GWZ} (Lemma 1.10) draws a related conclusion when the role of the pair $(S,\pi)$ is taken by a suitably transverse pair of algebraic varieties $(Z_1,Z_2)$. A similar phenomenon is evident in \cite{CIW} (Corollary 1.7). See Section \ref{SS2} below for further clarification.

None of the results or arguments in this paper require that the measure $\mu$ be nonnegative. When $S$ is a hypersurface at least, admissible measures $\mu$ arise from certain \textit{generalised Wigner distributions} $\nu$ on the tangent bundle of $S$. These are of particular interest in optics (see \cite{Al}) and are rarely nonnegative. In these situations one actually has the \textit{pointwise} identity $|\widehat{gd\sigma}|^2=X^*\mu$, which sheds some rather different light on the Stein type conjecture \eqref{genSt}, but falls short of providing a resolution due to the lack of positivity. We refer the interested reader to  \cite{BGNO}.

We conclude this section with three concrete applications that serve to further contextualise Theorem \ref{mainid}. 
\subsection{Application 1: Restriction-Brascamp--Lieb inequalities}\label{SS1}
Theorem \ref{mainid}, combined with a well-known theorem of Barthe \cite{Ba}, is easily seen to imply the endpoint case of the \textit{restriction-Brascamp--Lieb inequality} (\cite{BBFL}, \cite{Zhang}, \cite{BB}), in the so-called rank-1 case. The restriction-Brascamp--Lieb inequality is a broad generalisation of the multilinear restriction inequality of \cite{BCT}, involving a collection of extension operators associated with submanifolds of various dimensions. Concretely, suppose that for each $1\leq j\leq m$, $\Sigma_j:U_j\rightarrow\mathbb{R}^n$ is a smooth parametrisation of a $n_j$-dimensional submanifold $S_j$ of $\mathbb{R}^n$ by a neighbourhood $U_j$ of the origin in $\mathbb{R}^{n_j}$, and let
$$
E_jg_j(\xi):=\int_{U_j}e^{-2\pi i\xi\cdot\Sigma_j(x)}g_j(x)dx
$$
be the associated (parametrised) extension operator. In this setting it is natural to conjecture that if the Brascamp--Lieb constant $\mbox{BL}(\linearfamily,\mathbf{p})$ is finite for the linear maps $\linearmapj:=(\mathrm{d}\Sigma_j(0))^*:\mathbb{R}^n\rightarrow\mathbb{R}^{n_j}$, then provided the neighbourhoods $U_j$ of $0$ are chosen to be small enough, the inequality
\begin{equation}\label{genmultrest}
\int_{\mathbb{R}^n}\prod_{j=1}^m|E_jg_j|^{2p_j}\lesssim\prod_{j=1}^m\|g_j\|_{L^2(U_j)}^{2p_j}
\end{equation}
holds for all $g_j\in L^2(U_j)$, $1\leq j\leq m$. We note that the weaker local inequality
\begin{equation*}
\int_{B(0,R)}\prod_{j=1}^m|E_jg_j|^{2p_j}\leq C_\varepsilon R^\varepsilon\prod_{j=1}^m\|g_j\|_{L^2(U_j)}^{2p_j},
\end{equation*}
involving an arbitrary $\varepsilon>0$ loss was established in \cite{BBFL} (see also \cite{BEsc} and \cite{Zhang} where the power loss is reduced to polylogarithmic). Here we make a modest contribution to this problem using a simple instance of Theorem \ref{mainid}.
\begin{corollary}
The global  inequality \eqref{genmultrest} holds whenever $n_1=\cdots=n_m=1$.
\end{corollary}
\begin{proof}
We begin by observing that the linear map $\linearmapj:=(\mathrm{d}\Sigma_j(0))^*:\mathbb{R}^n\rightarrow\mathbb{R}$ satisfies $\linearmapj x=\langle x, v_j\rangle$, where $v_j$ is a tangent vector to $S_j$ at the point $\Sigma_j(0)$. By Barthe's finiteness characterisation \cite{Ba}, the Brascamp--Lieb constant $\mbox{BL}(\linearfamily,\mathbf{p})$ is finite if and only if $\mathbf{p}$ belongs to the convex hull of the points $\mathbf{p}^J$, where $J$ denotes a subset of $[1,m]:=\{1,\hdots,m\}$ of cardinality $n$, the exponent $p^J_j=\mathbf{1}_J(j)$, and $J$ is such that the set $\{v_j:j\in J\}$ forms a basis for $\mathbb{R}^n$. Here $\mathbf{1}_J:[0,m]\rightarrow \{0,1\}$ denotes the indicator function of $J$. 
Now, since $\mbox{BL}(\linearfamily,\mathbf{p})$ is finite, there are nonnegative scalars $\lambda_J$ such that 
$$
\mathbf{p}=\sum_{J}\lambda_J \mathbf{p}^J\;\;\mbox{ and }\;\;\sum_J\lambda_J=1.
$$
By the $m$-linear H\"older inequality it follows that
\begin{eqnarray*}
\begin{aligned}
\int_{\mathbb{R}^n}\prod_{j=1}^m|E_jg_j|^{2p_j}=\int_{\mathbb{R}^n}\prod_J\Bigl(\prod_{j=1}^m|E_jg_j|^{2p^J_j}\Bigr)^{\lambda_J}=
\int_{\mathbb{R}^n}\prod_J\Bigl(\prod_{j\in J}|E_jg_j|^{2}\Bigr)^{\lambda_J}
\leq\prod_{J}\Bigl\|\prod_{j\in J}E_jg_j\Bigr\|_2^{2\lambda_J},
\end{aligned}
\end{eqnarray*}
and so it remains to show that
\begin{equation}\label{remains}
\Bigl\|\prod_{j\in J}E_jg_j\Bigr\|_2\lesssim\prod_{j\in J}\|g_j\|_2
\end{equation}
for each $J$. Provided the neighbourhoods $U_j$ are chosen small enough, this elementary inequality is a straightforward consequence of Theorem \ref{mainid} (in its parametrised form \eqref{idpara}) applied to the cartesian product $$S:=\prod_{j\in J} S_j\subset (\mathbb{R}^n)^n\cong \mathbb{R}^{n^2}$$ 
with the (diagonal) subspace $\pi=\{(x_1,\hdots,x_n)\in (\mathbb{R}^n)^n: x_1=\cdots=x_n\}$. It is also instructive to provide a direct proof of \eqref{remains} as it merely requires a change of variables and Plancherel's theorem.
\end{proof}

\subsection{Application 2: Weighted $L^2$ norm identities}\label{Sect:stein}\label{SS2}
As we have already discussed, Theorem \ref{mainid} sheds some light on the very general formulation of Stein's conjecture for extension operators \eqref{genSt}.

In the particular case that $w$ lies in the image of $T_{k,n}^*$, and satisfies a certain transversality condition, the conjectural inequality \eqref{genSt} follows from Theorem \ref{mainid}, where it is in fact \textit{an identity}.  
\begin{corollary} If $w=T_{k,n}^*u$ for some $u:\M_{k,n}\rightarrow [0,\infty)$ whose support is transverse to $S$ in the sense of \rm{(T)} and \rm{(GT)}, then
\begin{equation}\label{genStid}
\int_{\mathbb{R}^n}|\widehat{gd\sigma}(x)|^2w(x)dx=\int_S |g(\xi)|^2\sup_{y\in T_\xi S}T_{n-k,n}w((T_\xi S)^\perp,y)d\sigma(\xi).
\end{equation}
\end{corollary}
A similar observation was made in the case $S=\mathbb{S}^{n-1}$ in \cite{BN}; see also \cite{CIW} for some related results.
\begin{proof}
Observe first that by an application of \eqref{id},
$$
\int_{\mathbb{R}^n}|\widehat{gd\sigma}|^2w=\langle T_{k,n}(|\widehat{gd\sigma}|^2),u\rangle=\langle T_{k,n}T_{n-k,n}^*\mu, u\rangle
=\int_{\M_{k,n}}T_{n-k,n}w\;d\mu.
$$
Here $\mu$ may be any admissible measure on $\M_{n-k,n}$ -- that is, the pushforward of a measure $\nu$ on $TS$ satisfying the marginal condition \eqref{marg}; for example $d\nu(\xi, y)=|g(\xi)|^2\delta_0(y)dyd\sigma(\xi)$.
By the transversality hypotheses on the support of $u$, it follows that $T_{n-k,n}w(\pi,y)$ is independent of $y$, and so
\begin{eqnarray*}
\begin{aligned}
\int_{\mathbb{R}^n}|\widehat{gd\sigma}|^2w&=\int_{\M_{k,n}}\sup_{z\in\pi^\perp}\left(T_{n-k,n}w(\pi,z)\right)d\mu(\pi,y)\\&=\int_{TS}\sup_{z\in T_\xi S}\left(T_{n-k,n}w((T_\xi S)^\perp,z)\right)d\nu(\xi,y)\\&=\int_S |g(\xi)|^2\sup_{y\in T_\xi S}T_{n-k,n}w((T_\xi S)^\perp,y)d\sigma(\xi),
\end{aligned}
\end{eqnarray*}
by the definition of $\mu$ and the marginal condition \eqref{marg}. 
\end{proof}

\subsection{Application 3: Convolution identities for extension operators}\label{SS3}
Theorem \ref{mainid} is quickly seen to imply the following multilinear extension identity of Iliopoulou and the first author in \cite{BI}.
\begin{corollary}\label{cor:BI}
Suppose $S_1,\hdots,S_n$ are smooth codimension-1 submanifolds of $\mathbb{R}^n$. Suppose further that the volume form $v_1(\xi_1)\wedge\cdots\wedge v_n(\xi_n)$ is nonvanishing, where $v_j(\xi_j)$ is a unit normal to $S_j$ for each $\xi_j\in S_j$ and $1\leq j\leq n$. Then,
$$
|\widehat{g_1d\sigma_1}|^2*\cdots*|\widehat{g_nd\sigma_n}|^2
\equiv
\int_{S_1\times\cdots\times S_n}\frac{|g_1(\xi_1)|^2\cdots|g_n(\xi_n)|^2}{|v_1(\xi_1)\wedge\cdots\wedge v_n(\xi_n)|}d\sigma_1(\xi_1)\cdots d\sigma_n(\xi_n).
$$
\end{corollary}
\begin{proof}
By modulation invariance it suffices to prove the claimed identity at the origin. Observe that
$$
|\widehat{g_1d\sigma_1}|^2*\cdots*|\widehat{g_nd\sigma_n}|^2(0)=n^{-n/2}T_{n(n-1),n^2}(|\widehat{gd\sigma_S}|^2)(\pi,0),
$$
where $S=S_1\times\cdots\times S_n$, $g=g_1\otimes\cdots\otimes g_n$ and $$\pi=\Bigl\{x=(x_1,\hdots,x_n)\in(\mathbb{R}^n)^n:\sum_{j=1}^n x_j=0\Bigr\}.$$ Finally, a straightforward linear algebra argument reveals that
$$
|(T_\xi S)^\perp\wedge\pi|
=
n^{-n/2}
|v_1(\xi_1)\wedge\cdots\wedge v_n(\xi_n)|,
$$
allowing the desired conclusion to follow from Theorem \ref{mainid}.
\end{proof}
As may be expected Corollary \ref{cor:BI} may be generalised to submanifolds $S_1,\hdots,S_m$ or varying dimensions, provided certain natural arithmetic and transversality conditions are satisfied.
\section{The proof of Theorem \ref{mainid}}\label{sect:pf}
We begin with the first identity in \eqref{id}, and use \eqref{graphoverpi} to write
$$
\widehat{gd\sigma}(x)=\int_U e^{ix\cdot(\Sigma(u))}g(\Sigma(u))J(u)d\lambda_\pi(u),
$$
where $J(u)=\left|\frac{\partial\Sigma}{du_1}\wedge\cdots\wedge\frac{\partial\Sigma}{\partial u_k}\right|$ and $\Sigma(u)=u+\phi(u)$.
Consequently, by Plancherel's theorem on $\pi$,
\begin{align*}
T_{k,n}(|\widehat{gd\sigma}|^2)(\pi,y)
&=\int_\pi|\widehat{gd\sigma}(x+y)|^2d\lambda_\pi(x)\\&=
\int_{\pi}\left|\int_U e^{-2\pi i(x\cdot u+ y\cdot\phi(u))}g(u+\phi(u))J(u)du\right|^2d\lambda_\pi(x)\\
&=\int_\pi \left|e^{-2\pi iy\cdot\phi(u)}g(u+\phi(u))J(u)\right|^2d\lambda_\pi(u)\\
&=\int_\pi \left|g(u+\phi(u))J(u)^{1/2}\right|^2J(u)d\lambda_\pi(u)\\
&=\int_S |g(\xi)|^2J(u(\xi))d\sigma(\xi),
\end{align*}
where $u(\xi)$ is the orthogonal projection of $\xi\in S$ onto $\pi$. 
Since $|(T_{\Sigma(u)} S)^\perp\wedge\pi|=|T_{\Sigma(u)} S\wedge\pi^\perp|$ it therefore remains to show that
\begin{equation}\label{lastbit}
J(u)=\frac{1}{|T_{\Sigma(u)} S\wedge\pi^\perp|}.
\end{equation}
In order to establish \eqref{lastbit} we may, by the rotation-invariance of the statement of Theorem \ref{mainid}, assume that $\pi=\langle e_1,\hdots,e_k\rangle$, where $e_1,\hdots, e_n$ denote the standard basis vectors of $\mathbb{R}^n$. We observe first that since $\phi:\pi\rightarrow\pi^\perp$,
\begin{equation}\label{one}
\left|\frac{\partial\Sigma}{\partial u_1}\wedge\cdots\wedge\frac{\partial\Sigma}{\partial u_k}\wedge e_{k+1}\wedge\cdots\wedge e_n\right|=1.
\end{equation}
Next we construct orthogonal (as opposed to orthonormal) $v_1,\hdots,v_k\in T_{\Sigma(u)}S$ from $\frac{\partial\Sigma}{\partial u_1},\hdots, \frac{\partial\Sigma}{\partial u_k}$ by the Gram--Schmidt process, and observe that $$v_1\wedge\cdots \wedge v_k=\frac{\partial\Sigma}{\partial u_1}\wedge\cdots\wedge\frac{\partial\Sigma}{\partial u_k}.$$ Consequently,
$$
|v_1|\cdots |v_k|=\left|\frac{\partial\Sigma}{\partial u_1}\wedge\cdots\wedge\frac{\partial\Sigma}{\partial u_k}\right|,
$$
and so by \eqref{one},
$$
|T_{\Sigma(u)} S\wedge\pi^\perp|=\left|\frac{v_1}{|v_1|}\wedge\cdots\wedge\frac{v_k}{|v_k|}\wedge e_{k+1}\wedge\cdots\wedge e_n\right|=\frac{1}{J(u)}.
$$

Turning to the second identity in \eqref{id}, by \eqref{e:KPlane} we have
\begin{align*}
T_{k,n}T_{n-k,n}^* \mu (\pi,y)
&= 
\int_{\mathbb R^n} T^*_{n-k,n}\mu(x)\delta_{\pi+\{y\}}(x)dx\\
&=
\int_{\mathbb R^n}\left(\int_{\M_{n-k,n}}\delta_{\theta+\{z\}}(x)d\mu(\theta,z)\right)\delta_{\pi+\{y\}}(x)dx\\
&=
\int_{\M_{n-k,n}}
\left(
\int_{\mathbb R^n} 
\delta_{\pi+\{y\}}(x)
\delta_{\theta+\{z\}}(x)
dx
\right)
d\mu(\theta,z).
\end{align*}
Recalling that the transverse subspaces $\pi$ and $\theta$ are $k$-dimensional and $(n-k)$-dimensional respectively, an affine change of variables reveals that
\begin{equation}\label{angle}
\int_{\mathbb R^n} 
\delta_{\pi+\{y\}}(x)
\delta_{\theta+\{z\}}(x)
dx=\frac{1}{|\theta\wedge \pi|}
\end{equation}
for all $y\in\pi^\perp$ and $z\in\theta^\perp$. Hence,
\begin{align*}
T_{k,n}T_{n-k,n}^* \mu (\pi,y)&=
\int_{\M_{n-k,n}}\frac{d\mu(\theta,z)}{|\theta\wedge \pi|}\\
&=
\int_{TS}\frac{d\nu(\xi,y')}{|(T_\xi S)^\perp\wedge \pi|}\\
&=
\int_{S}\frac{d\nu_S(\xi)}{|(T_\xi S)^\perp\wedge\pi|}\\
&=
\int_{S}\frac{|g(\xi)|^2}{|(T_\xi S)^\perp\wedge\pi|}d\sigma(\xi),
\end{align*}
where in the final line we have used \eqref{marg}.
\section{Necessity of the transversality hypotheses}\label{sect:opt}

While the hypothesis (T) is clearly necessary for finiteness in \eqref{id}, the necessity of the global hypothesis (GT) is rather less aparent. Here we establish that if the data $T_k(|\widehat{gd\sigma}|^2)(\pi,y)$ is independent of $y$, as claimed by \eqref{id}, then (GT) must hold. To this end we fix a $k$-dimensional $C^1$ submanifold $S$ and a $k$-dimensional subspace $\pi$, and
assume that (GT) fails, that is, there exist $\xi_0,\eta_0 \in S$ such that $\xi_0 - \eta_0 \in \pi^\perp\backslash\{0\}$. 
	We then define $g_0:S\rightarrow\mathbb{R}$ by
	$
	g_0 = \chi_{C_\varepsilon(\xi_0)} + \chi_{C_\varepsilon(\eta_0)},
	$
	where $C_\varepsilon(a) := S\cap \{ \xi \in \mathbb{R}^n: |\xi - a|< \varepsilon\}$, and  $\varepsilon>0$ will be chosen later. 
	It will suffice to show that the directional derivative
	\begin{equation}\label{e:dirder}
	(\xi_0-\eta_0)\cdot\nabla_y \big[ T_{k,n}(|\widehat{g_0d\sigma}|^2)(\pi,\cdot) \big](0)
	=
	(\xi_0-\eta_0)\cdot\nabla_y \big[ \int_{\pi} |\widehat{g_0d\sigma}(x+y)|^2 d\lambda_{\pi}(x) \big]|_{y=0}
	\end{equation}
does not vanish for some $\varepsilon>0$.
	We remark that 
	\begin{equation}\label{e:Finite}
	\int_{\pi} |\widehat{gd\sigma}(x+y)|^2 d\lambda_{\pi}(x) < \infty 
	\end{equation}
	for all $y \in \pi^\perp$ and all $g \in L^2(S)$ thanks to (T). 
	In particular, \eqref{e:Finite} holds for $g=g_0$ and $g(\xi) = (\xi_0-\eta_0)\cdot \xi g_0(\xi)$, from which it follows that 
	\begin{equation}\label{e:FiniteDeri}
		\int_{\pi} \big|(\xi_0-\eta_0)\cdot\nabla_{\mathbb{R}^n} \big( |\widehat{g_0d\sigma}|^2 \big) (x)\big| d\lambda_{\pi}(x) < \infty.
	\end{equation}
	This justifies an interchange of derivative and integral in \eqref{e:dirder}, reducing matters to establishing that 
	\begin{equation}\label{e:dirder2}
		\int_{\pi} (\xi_0-\eta_0)\cdot\nabla_{\mathbb{R}^n} \big( |\widehat{g_0d\sigma}|^2 \big) (x) d\lambda_{\pi}(x)\not=0
	\end{equation}
for some $\varepsilon>0$.
	By the dominated convergence theorem, \eqref{e:FiniteDeri} also ensures that 
	\begin{equation}\label{e:ApproxDec}
		(\xi_0-\eta_0)\cdot\nabla_y \big[ \int_{\pi} |\widehat{g_0d\sigma}(x+y)|^2\, d\lambda_{\pi}(x) \big] |_{y=0}
		=
		\lim_{R\to \infty} 
		\int_{\pi} (\xi_0-\eta_0)\cdot\nabla_{\mathbb{R}^n} \big( |\widehat{g_0d\sigma}|^2 \big) (x)\gamma_R(x) d\lambda_{\pi}(x),
	\end{equation} 
	where $\gamma_R(x)=\gamma(x/R)$ for some Schwartz function $\gamma$ on $\pi$ chosen to have compact Fourier support and satisfy $\gamma(0)=1$.
		For sufficiently large $R$, by Fubini's theorem we have
	\begin{align*}
		\int_{\pi} (\xi_0-\eta_0)\cdot\nabla_{\mathbb{R}^n}& \big( |\widehat{g_0d\sigma}|^2 \big) (x)\gamma_R(x) d\lambda_{\pi}(x)\\
		&= 
		-2\pi i 
		\int_{S}\int_{S} (\xi_0 - \eta_0) \cdot 
		(\xi-\eta)
		g_0(\xi)g_0(\eta)
		\widehat{\gamma}_R(P_{\pi}(\xi-\eta))
		d\sigma(\xi) d\sigma(\eta).
	\end{align*}
	We now claim that 
	$$
		\int_{S}\int_{S} (\xi_0 - \eta_0) \cdot 
		(\xi-\eta)
		g_0(\xi)g_0(\eta)
		\widehat{\gamma}_R(P_{\pi}(\xi-\eta))
		d\sigma(\xi) d\sigma(\eta)
	$$
	is bounded away from zero for sufficiently large $R$, provided $\varepsilon$ is chosen small enough. 
	By our choice of $g_0$ we now write
	\begin{align*}
		&
		\int_{S}\int_{S} (\xi_0 - \eta_0) \cdot 
		(\xi-\eta)
		g_0(\xi)g_0(\eta)
		\widehat{
		\gamma}_R
		(P_{\pi}(\xi-\eta))
		d\sigma(\xi) d\sigma(\eta)	
		=
		2I_{\xi_0,\eta_0}(R)
		+ I_{\xi_0,\eta_0}'(R),
	\end{align*}
	where 
	\begin{align*}
		I_{\xi_0,\eta_0}(R)
		&:= 
		\int_{S}\int_{S} (\xi_0 - \eta_0) \cdot 
		(\xi-\eta)
		\chi_{C_\varepsilon(\xi_0)}(\xi) \chi_{C_\varepsilon(\eta_0)}(\eta)
		\widehat{
		\gamma}_R
		(P_{\pi}(\xi-\eta))
		d\sigma(\xi) d\sigma(\eta). 
	\end{align*}
	By (T) we may choose $\varepsilon>0$ so that (GT) holds on $C_\varepsilon(\xi_0)$ and $C_\varepsilon(\eta_0)$. 
	In particular, for $\zeta_0$ being $\xi_0$ or $\eta_0$, we have 
	$$
	\int_{C_\varepsilon(\zeta_0)} \varphi(\xi) d\sigma(\xi)
	= 
	\int_{{U}_{\varepsilon,\zeta_0}} \varphi(u+ \phi_{\varepsilon,\zeta_0}(u))J_{\varepsilon,\zeta_0}(u) du,
	$$
	for some open $U_{\varepsilon,\zeta_0} \subset \pi$ and $\phi_{\varepsilon,\zeta_0}: {U}_{\varepsilon,\zeta_0} \to \pi^\perp$, and any test function $\varphi$.
	Note that the Jacobians are bounded from above and below. Since $\xi_0-\eta_0 \in \pi^\perp$, it is straightforward to verify that 
	$
	| I_{\xi_0,\eta_0}'(R) | \lesssim_{\varepsilon,\xi_0,\eta_0} R^{-1}. 
	$
	Hence it suffices to show  
	\begin{equation}\label{e:GoalDec1-2}
		|I_{\xi_0,\eta_0}(R)| > 0
	\end{equation}
	uniformly in sufficiently large $R$. 
	Since $\varepsilon$ is small, 
	$
	(\xi_0 - \eta_0) \cdot 
		(\xi-\eta)
		\sim 
		|\xi_0 - \eta_0|^2
	$ for all $\xi \in C_{\varepsilon}(\xi_0) $ and all $\eta \in C_{\varepsilon}(\eta_0)$, 
	and hence, using the fact that $\xi_0 - \eta_0 \in \pi^\perp $, 
	\begin{align*}
	I_{\xi_0,\eta_0}(R)
		&\sim 
		| \xi_0 - \eta_0|^2
		\int_{C_\varepsilon(\xi_0)}\int_{C_\varepsilon(\eta_0)}
		\widehat{
		\gamma}_R (P_{\pi}(\xi-\eta))
		d\sigma(\xi) d\sigma(\eta)\\
		&= 
		| \xi_0 - \eta_0|^2
		\int_{C_\varepsilon(\xi_0)-\xi_0}\int_{C_\varepsilon(\eta_0)-\eta_0}
		\widehat{
		\gamma}_R(P_{\pi}(\xi- \eta ))
		d\sigma(\xi) d\sigma(\eta),
	\end{align*}
	where $d\sigma$ continues to denote surface measure. 
	If $u_0 \in {U}_{\varepsilon,\xi_0}$ is such that $\xi_0 = u_0 + \phi_{\varepsilon,\xi_0}(u_0)$, and $v_0 \in {U}_{\varepsilon,\eta_0}$ is such that $\eta_0 = v_0 + \phi_{\varepsilon,\eta_0}(v_0)$, then 
	$$
	C_\varepsilon(\xi_0)-\xi_0
	=
	\{u' + \phi_{\varepsilon, \xi_0}(u'+u_0) - \phi_{\varepsilon, \xi_0}(u_0): u' \in {U}_{\varepsilon,\xi_0}'
	\},
	$$
	where ${U}_{\varepsilon,\xi_0}':= {U}_{\varepsilon,\xi_0} - u_0$, and similarly 
	$$
	C_\varepsilon(\eta_0)-\eta_0
	=
	\{v' + \phi_{\varepsilon, \eta_0}(v'+v_0) - \phi_{\varepsilon, \eta_0}(v_0): v' \in {U}_{\varepsilon,\eta_0}'
	\},
	\;\;\;
	{U}_{\varepsilon,\eta_0}':= {U}_{\varepsilon,\eta_0} - v_0.
	$$
	Since $\phi_{\varepsilon, \xi_0}(u'+u_0) - \phi_{\varepsilon, \xi_0}(u_0)$ and $ \phi_{\varepsilon, \eta_0}(v'+v_0) - \phi_{\varepsilon, \eta_0}(v_0)$ belong to $ \pi^\perp$,  
	$$
	P_{\pi}(\xi-\eta)
	=
	P_\pi 
	\big( 
	u' + \phi_{\varepsilon, \xi_0}(u'+u_0) - \phi_{\varepsilon, \xi_0}(u_0)
	-
	(v' + \phi_{\varepsilon, \eta_0}(v'+v_0) - \phi_{\varepsilon, \eta_0}(v_0))
	\big)
	=
	u'-v' 
	$$
	for all $\xi \in C_{\varepsilon}(\xi_0)$ and  $\eta \in C_{\varepsilon}(\eta_0)$. 
	Hence, 
	\begin{align*}
	I_{\xi_0,\eta_0}(R)
		&\sim 
		| \xi_0 - \eta_0|^2
		\int_{{U}_{\varepsilon,\xi_0}'}
		\int_{{U}_{\varepsilon,\eta_0}'}
		\widehat{
		\gamma}_R(u'-v')
		du'dv'. 
	\end{align*}	
	Since ${U}_{\varepsilon,\xi_0}' $ and ${U}_{\varepsilon,\eta_0}'$ contain the origin,  
	\begin{align*}
	I_{\xi_0,\eta_0}(R)
		&\gtrsim 
		| \xi_0 - \eta_0|^2
		\int_{{U}_{\varepsilon,\xi_0}'  \cap {U}_{\varepsilon,\eta_0}'}\int_{{U}_{\varepsilon,\xi_0}'  \cap {U}_{\varepsilon,\eta_0}'}
		\widehat{
		\gamma}_R(u'-v') du'dv'
		\sim_\varepsilon |\xi_0-\eta_0|^2, 
	\end{align*}
	from which \eqref{e:GoalDec1-2} follows.

\section*{Acknowledgment}
This work is supported by JSPS Overseas Research Fellowship and  JSPS Kakenhi grant numbers 19K03546, 19H01796 and 21K13806 (Nakamura), JSPS Kakenhi grant numbers 	20J11851, JSPS Overseas Challenge Program for Young Researchers, Daiwa Foundation Small Grant, and funds from the Harmonic Analysis Incubation Research Group at Saitama University (Shiraki).  The third author would also like to thank Susana Guti\'{e}rrez for several discussions, and the University of Birmingham, where part of this research was conducted, for their kind hospitality.


\begin{thebibliography}{MMMMM}
\bibitem{Al} M. A. Alonso, \textit{Wigner functions in optics: describing beams as ray bundles and pulses as particle ensembles}, Advances in Optics and Photonics \textbf{3}, Issue 4 (2011), 272--365.

\bibitem{BBC3} J. A. Barcel\'{o}, J. Bennett, A. Carbery, \textit{A note on localised weighted estimates for the extension operator}, J. Aust. Math. Soc. \textbf{84} (2008), 289--299.

\bibitem{BRV} J. A. Barcel\'{o}, A. Ruiz, L. Vega, \textit{Weighted estimates for the Helmholtz equation and consequences}, J. Funct. Anal. \textbf{150} (1997), 356--382.

\bibitem{Ba} F. Barthe,  \textit{On a reverse form of the Brascamp--Lieb inequality}, Invent. Math. \textbf{134}
(1998), 355--361.


\bibitem{BV} D. Beltran, L. Vega \textit{Bilinear identities involving the $k$-plane transform and Fourier extension operators}, Proc. Roy. Soc. Edinburgh Sect. A \textbf{150} (2020),  3349--3377.

\bibitem{BEsc} J. Bennett, \textit{Aspects of multilinear harmonic analysis related to transversality}, Harmonic analysis and partial differential equations, 1--28, Contemp. Math., 612, Amer. Math. Soc., Providence, RI, 2014.

\bibitem{BB} J. Bennett, N. Bez, \textit{Higher order transversality in harmonic analysis}, RIMS K\^oky\^uroku Bessatsu B88 (2021), 75--103.

\bibitem{BBFGI} J. Bennett, N. Bez, T. C. Flock, S. Guti\'errez, M. Iliopoulou, \textit{A sharp $k$-plane Strichartz estimate for the Schr\"odinger equation}, Trans. Amer. Math. Soc. \textbf{370} (2018), 5617--5633.

\bibitem{BBFL} J. Bennett, N. Bez, T. C. Flock, S. Lee, \textit{Stability of the Brascamp--Lieb constant and applications to harmonic analysis}, Amer. J. Math. \textbf{140} (2018), 543--569.

\bibitem{BCSV} J. Bennett, A. Carbery, F. Soria, A. Vargas, \textit{A Stein conjecture for the circle}, Math. Ann. \textbf{336} (2006), 671--695.

\bibitem{BCT} J. Bennett, A. Carbery, T. Tao, \textit{On the multilinear restriction and Kakeya conjectures}, Acta Math. \textbf{196} (2006), 261--302.

\bibitem{BGNO} J. Bennett, S. Gutierrez, S. Nakamura, I. Oliveira, \textit{A phase-space approach to weighted Fourier extension inequalities}, preprint 2024.

\bibitem{BI} J. Bennett, M. Iliopoulou, \textit{A multilinear extension identity on $\mathbb{R}^n$}, Math. Res. Lett. \textbf{25} (2018), 1089--1108.

\bibitem{BN} J. Bennett, S. Nakamura, \textit{Tomography bounds for the Fourier extension operator and applications}, Math. Ann. \textbf{380} (2021),  
119--159.





\bibitem{CIW} A. Carbery, M. Iliopoulou, H. Wang, to appear in Rev. Mat. Iberoam. 

\bibitem{GWZ} S. Guo, H. Wang, R. Zhang, \textit{A dichotomy for H\"ormander-type
oscillatory integral operators}, arXiv:2210.05851.



\bibitem{MT} S. Mizohata, \textit{On the Cauchy Problem}, Notes and Reports in Mathematics, Science and Engineering, \textbf{3}, Academic Press, San Diego, CA, 1985.

\bibitem{PV} F. Planchon, L. Vega, \textit{Bilinear virial identities and applications}, Ann. Scient. Ec. Norm. Sup., \textbf{42} (2009), 263--292.
\bibitem{S}  E. M. Stein, \textit{Some problems in harmonic analysis},
Proc. Sympos. Pure Math., Williamstown, Mass., (1978), 3--20.


\bibitem{Stovall} B. Stovall, \textit{Waves, Spheres, and Tubes.
A Selection of Fourier Restriction Problems,
Methods, and Applications}, Not. Amer. Math. Soc., \textbf{66} (2019), 1013--1022.

\bibitem{VegaEsc} L. Vega, \textit{Bilinear virial identities and oscillatory integrals}, Harmonic analysis and partial differential equations, 219--232, Contemp. Math., 505, Amer. Math. Soc., Providence, RI, 2010.

\bibitem{Zhang} R. Zhang, \textit{The endpoint perturbed Brascamp--Lieb inequalities with examples}, Anal. PDE. \textbf{11} (2018), 555--581.

\end{thebibliography}
\end{document}